\documentclass[reqno]{amsart}
\usepackage[utf8x]{inputenc}  

\usepackage{listings}
\usepackage{caption}
\usepackage{easytable}

\usepackage{setspace}
\usepackage[left=3.1cm, right=3.1cm, bottom=4cm]{geometry}                 

\usepackage{graphicx} 
\usepackage{pgf,tikz}
\usetikzlibrary{arrows}
\usepackage{amssymb}
\usepackage{pdfsync}
\usepackage{mathrsfs}
\usepackage{hyperref} 

\usepackage{epstopdf}
\usepackage{bbm} 
\usepackage[colorinlistoftodos,prependcaption,textsize=tiny]{todonotes}
\usepackage{xargs}
\usepackage{bm}
\usepackage{lmodern}

\usepackage{listings} 

\def\R{\mathbb{R}}

\def\N{\mathbb{N}}
\def\Z{\mathbb{Z}}
\def\C{\mathbb{C}}

\def\F{\mathbb{F}}

\def\P{\mathcal{P}}

\def\f{\widehat{f}}

\renewcommand{\d}{\text{\rm d}}

\newcommand{\bxi}{{\bm \xi}}
\newcommand{\bb}{{\bm b}}

\newcommand{\mc}{\mathcal}

\newtheorem{theorem}{Theorem}
\newtheorem{conjecture}[theorem]{Conjecture}

\newtheorem*{definition*}{Definition}

\newtheorem{lemma}[theorem]{Lemma}

\makeatletter
\DeclareFontFamily{U}{tipa}{}
\DeclareFontShape{U}{tipa}{m}{n}{<->tipa10}{}
\newcommand{\arc@char}{{\usefont{U}{tipa}{m}{n}\symbol{62}}}%

\makeatother

\numberwithin{equation}{section}

\allowdisplaybreaks

\newcommand{\intav}[1]{\mathchoice {\mathop{\vrule width 6pt height 3 pt depth  -2.5pt
\kern -8pt \intop}\nolimits_{\kern -6pt#1}} {\mathop{\vrule width
5pt height 3  pt depth -2.6pt \kern -6pt \intop}\nolimits_{#1}}
{\mathop{\vrule width 5pt height 3 pt depth -2.6pt \kern -6pt
\intop}\nolimits_{#1}} {\mathop{\vrule width 5pt height 3 pt depth
-2.6pt \kern -6pt \intop}\nolimits_{#1}}}

\newcommand{\intavl}[1]{\mathchoice {\mathop{\vrule width 6pt height 3 pt depth  -2.5pt
\kern -8pt \intop}\limits_{\kern -6pt#1}} {\mathop{\vrule width 5pt
height 3  pt depth -2.6pt \kern -6pt \intop}\nolimits_{#1}}
{\mathop{\vrule width 5pt height 3 pt depth -2.6pt \kern -6pt
\intop}\nolimits_{#1}} {\mathop{\vrule width 5pt height 3 pt depth
-2.6pt \kern -6pt \intop}\nolimits_{#1}}}

    \newcommand{\x}{{\bm x}}
     \newcommand{\n}{{\bm n}}
       
       \newcommand{\p}{{\bm p}}

    \renewcommand{\bxi}{{\bm \xi}}  
     \newcommand{\bzeta}{{\bm \zeta}}    
      \newcommand{\bell}{{\bm \ell}}

\title[Sharp  extension  for the moment curve on finite fields]{Sharp endpoint extension inequalities \\ for the moment curve on finite fields}

\author[Biswas]{Chandan Biswas}
\address{Department of Mathematics, Indian Institute of Technology Bombay, Mumbai 400076, India}
\email{cbiswas@iitb.ac.in}

\author[Carneiro]{Emanuel Carneiro}
\address{ICTP - The Abdus Salam International Centre for Theoretical Physics, 
Strada Costiera, 11, I - 34151, Trieste, Italy}
\email{carneiro@ictp.it}

\author[Flock]{Taryn C. Flock}
\address{Macalester College\\Mathematics, Statistics, and Computer Science\\
Olin-Rice Science Center, Room 222\\
Saint Paul, MN 55105-1899\\USA}
\email{tflock@macalester.edu}

\author[Oliveira e Silva]{Diogo Oliveira e Silva}
\address{ 
Center for Mathematical Analysis, Geometry and Dynamical Systems \&
Departamento de Matem\'{a}tica\\ 
Instituto Superior T\'{e}cnico\\
Av.\@ Rovisco Pais\\ 
1049-001 Lisboa, Portugal} 
\email{diogo.oliveira.e.silva@tecnico.ulisboa.pt}

\author[Stovall]{Betsy Stovall}
\address{University of Wisconsin--Madison\\Department of Mathematics\\
480 Lincoln Drive\\ 
Madison, WI 53706\\
USA}
\email{stovall@math.wisc.edu}

\author[Tautges]{James Tautges}
\address{University of Wisconsin--Madison\\Department of Mathematics\\
480 Lincoln Drive\\ 
Madison, WI 53706\\
USA}
\email{tautges2@wisc.edu }

\begin{document}

\begin{abstract}
We investigate the sharp endpoint extension inequality for the moment curve in finite fields. We determine  the optimal  constant and characterize the  maximizers in two complementary regimes: (i) low dimensions $d\leq 20$; (ii) large field cardinality $q\geq \frac{d(d-1)}{2 \log 6} + \frac{(2d-1)}{3}
$. Our proof strategy relies on  an intriguing interplay between analysis, algebra and combinatorics.
\end{abstract}

\maketitle 

\section{Introduction}
\subsection{Background} 

The moment curve $\{\gamma(t):=(t,t^2,\ldots,t^d):\, t\in\R\}\subset\R^d$ is a simple
object with a large  group of symmetries whose Fourier restriction problem was  solved forty years ago:  Drury \cite{Dr85} showed that the adjoint restriction, or extension, operator to the moment curve,
\[\mathcal Ef(\x):=\int_\R e^{i\x\cdot\gamma(t)} f(t)\,\textup d t,\]
extends to a bounded linear operator from $L^r(\R)$ to $L^s(\R^d)$ if and only if $s>\frac{d^2+d+2}2$ and $s=\frac{d(d+1)}2r'$, in all dimensions $d\geq 2.$
For every exponent pair $(r,s)$ in this range, Biswas--Stovall \cite{BS23} established  the existence of {\it maximizers}, i.e., non-zero functions $f$ for which $\|\mathcal E f\|_s = \|\mathcal E \|_{r\to s} \|f\|_r$.
However, the nature of these functions remains elusive, as it is not even  clear what the candidates should be.
For the state of the art in sharp restriction theory in $\R^d$, we refer to the recent surveys \cite{FOS17, NOST23}.
\smallskip

The extension problem on finite fields was first considered in 2002 by Mockenhaupt--Tao \cite{MT}, who established several extension inequalities on paraboloids and cones. Interestingly, they also  provided a complete solution to the finite field extension problem for the moment curve  $\Gamma:=\{\gamma(\xi)=(\xi,\xi^2,\ldots, \xi^d),\,\xi\in\mathbb F_q\}$, whenever the characteristic of the field $\mathbb F_q$ is strictly larger than $d$. More precisely, the authors of \cite{MT} showed that the operator norm ${\bf R}^\ast_\Gamma(r\to s)$ in the inequality 
\[\|(f\sigma_\Gamma)^\vee\|_{L^s(\mathbb F_q^d,\textup d \x)} \leq {\bf R}^{\ast}_\Gamma(r\to s) \|f\|_{L^r(\Gamma,\textup d \sigma)}\]
can be bounded independently of the underlying field $\mathbb F_q$ if and only if $s\geq \max\{2d,dr'\}$. The crucial step turns out to be the endpoint estimate corresponding to ${\bf R}^\ast_\Gamma(2\to 2d)$. The measures $\d \x$ in the physical space and $\d\sigma$ along the curve $\Gamma$ in the frequency space  will be properly defined in \S \ref{sec_setup} below.
\smallskip

Sharp restriction theory and the finite field extension problem have received a great
deal of attention in the last two decades, but they intersected only very recently: González-Riquelme and Oliveira e Silva \cite{CGRDOS} proved  that constant functions maximize the extension inequality from the parabola $\mathbb{P}^1\subset \mathbb{F}^{2\ast}_q$ and the paraboloid $\mathbb{P}^2\subset \mathbb{F}_q^{3\ast}$ at the correspondent euclidean Stein--Tomas endpoints, and identified gaussians as the full set of  maximizers for the $L^2\to L^4$ extension inequality from $\mathbb{P}^2$ whenever $q$ is congruent to 1 modulo 4. They  established further   results on cones and hyperbolic paraboloids (i.e., saddles), but the higher codimensional case remained uncharted territory. On the other hand, (non-sharp) estimates for averaging operators for a few model cases, including the moment curve, have been considered in finite fields \cite{Carbery2008}.
\smallskip

In the present paper, we establish the first sharp endpoint extension inequalities  for the moment curve on finite fields.
\smallskip

The moment curve has featured prominently in other problems in analysis and number theory. One of the major advances in analytic number theory in the last decade has been the resolution by Bourgain, Demeter and Guth \cite{BDG16} of the general case of Vinogradov's mean value conjecture, a key statement in the theory of Waring's problem. This is implied by the sharp $\ell^2$ decoupling inequality for the moment curve, whose proof uses a multilinear variant of the decoupling
inequality, which in turn crucially relies on multilinear Kakeya--Brascamp--Lieb
type inequalities; see  \cite{GLYZK21} for a short proof, which instead uses  a bilinear variant of the decoupling
inequality.
Moment curves have also been used in applications to discrete geometry, e.g., cyclic polytopes \cite[Lemma 5.4.2]{Mat02}, the {\it no-three-in-line} problem (credited to P. Erdös by K. Roth in \cite{Ro51}), and a geometric proof of the chromatic number of Kneser graphs \cite[Section 3.5]{Mat03}.\smallskip

\subsection{Setup}\label{sec_setup} We generally follow the notation of Mockenhaupt--Tao \cite{MT} for the Fourier restriction/extension theory over finite fields. Throughout the text, we let $p$ be an odd prime and let $q = p^n$ for some $n \in \N$. We also let $\F_q$ denote the finite field with $q$ elements. 

\smallskip

 Let us  recall the basic definitions for Fourier analysis on the $d$-dimensional vector space $\F_q^d$. We denote by $\F_q^{d\ast}$ the dual of $\F_q^d$. The spaces $\F_q^{d}$ and $\F_q^{d\ast}$ are isomorphic, but their underlying natural measures are different. We endow the physical space $\F_q^{d}$ with the counting measure $\d \x$ and the frequency space $\F_q^{d\ast}$  with the normalized counting measure $\d \bxi$ so that $\F_q^{d\ast}$ has total mass $1$, i.e., 
\begin{align*}
\int_{\F_q^{d}} f(\x) \,\d\x := \sum_{\x \in \F_q^{d}} f(\x) \ \ \ {\rm and} \ \ \  \int_{\F_q^{d\ast}} g(\bxi) \,\d\bxi := \frac{1}{q^d}\sum_{\bxi \in \F_q^{d\ast}} g(\bxi). 
\end{align*}
Given a function $f: \F_q^{d} \to \C$, its Fourier transform $\f: \F_q^{d\ast} \to \C$ is defined as
\begin{align*}
\f(\bxi):= \int_{\F_q^{d}} f(\x) \,e(-\x \cdot \bxi)\, \d\x = \sum_{\x \in \F_q^{d}} f(\x) \,e(-\x \cdot \bxi)\,,
\end{align*}
where $e(y) := {\rm exp}(2 \pi i {\rm Tr}(y) /p)$ and the trace function ${\rm Tr}: \F_q \to \F_p$ is defined by ${\rm Tr}(y) := y + y^p + y^{p^2} + \ldots + y^{p^{n-1}}$. Here, as usual, $\x \cdot \bxi:= \sum_{i=1}^d x_i \xi_i$ if $\x = (x_1, \ldots, x_d)$ and $\bxi = (\xi_1, \ldots, \xi_d)$. Given $g:\F_q^{d\ast} \to \C$, its inverse Fourier transform $g^{\vee}:\F_q^{d} \to \C$ is defined as
\begin{align*}
g^{\vee}(\x):= \int_{\F_q^{d\ast}} g(\bxi) \,e(\x \cdot \bxi)\, \d\bxi = \frac{1}{q^d} \sum_{\bxi \in \F_q^{d\ast}} g(\bxi) \,e(\x \cdot \bxi).
\end{align*}
With this setup, we have Plancherel's identity
\begin{equation*}
\|f\|_{L^{2}(\F_q^d, \d \x)} = \|\f\|_{L^{2}(\F_q^{d\ast}, \d \bxi)}\,,
\end{equation*}
and Fourier inversion reads
$$
f(\x) = (\f)^{\vee} (\x)
$$
for each $\x \in \F_q^{d}$. The Fourier transform interchanges convolution and multiplication by 
\begin{align*}
\widehat{f_1}  \widehat{f_2} = \widehat{f_1 * f_2} \ \ \ {\rm and} \ \ \ \widehat{f_1f_2} = \widehat{f_1}* \widehat{f_2}.
\end{align*}
Note here that convolution in $\F_q^{d}$ is defined with respect to the counting measure $\d \x$, whereas convolution in $\F_q^{d\ast}$ is defined with respect to the normalized counting measure $\d \bxi$.

\smallskip

Given $d\geq 2$, let  $\gamma:\F_q \to \F_q^{d\ast}$ be given by 
$\gamma(\xi) = (\xi, \xi^2, \ldots, \xi^d),
$
and  $\Gamma = \{ \gamma(\xi)\, : \, \xi \in \F_q\} \subset \F_q^{d\ast}$ be the moment curve . Let $\sigma= \sigma_{\Gamma}$ be the normalized counting measure supported on $\Gamma$. Henceforth we assume  the characteristic $p$ of the underlying field to be strictly greater than the dimension $d$. Given $f: \Gamma \to \C$, the extension operator $(f \sigma)^{\vee}:\F_q^{d} \to \C$  is defined by 
\begin{align}\label{20250801_17:50}
(f \sigma)^{\vee}(\x) := \frac{1}{q} \sum_{\bxi \in \Gamma} f(\bxi) \, e(\x \cdot \bxi).
\end{align}
We seek to find the sharp endpoint inequality 
\begin{align}\label{20250526_09:29}
\|(f \sigma)^{\vee}\|_{L^{2d}(\F_q^d, \d \x)} \leq {\bf R}^\ast_\Gamma(2\to 2d) \,\|f\|_{L^{2}(\Gamma, \d\sigma)}\,,
\end{align}
identifying the optimal constant ${\bf R}^\ast_\Gamma(2\to 2d)$ and its maximizers.

\subsection{Main results} Our main results are closely related to certain combinatorial entities involving the classical problem of integer partitions. Let 
$$
\P_d := \{\bell = (\ell_1, \ell_2, \ldots, \ell_d ) \in \Z^d \ : \ \ell_1 \geq \ell_2 \geq \ldots \geq \ell_d \geq 0 \ \ {\rm and} \ \ \ell_1 + \ell_2 + \ldots + \ell_d = d\}
$$
denote the set of partitions of the integer  $d\geq 2$. For each $\bell \in \P_d$, let 
$$
\bb(\bell) := (b_0(\bell), b_1(\bell), b_2(\bell), \ldots, b_d(\bell)),
$$ 
where $b_j(\bell) := |\{ 1 \leq i \leq d \ : \ \ell_i = j\}|$ for $j \in\{ 1,2,\ldots,d\}$ and $b_0(\bell)$ is slightly different\footnote{This is just a minor choice of notation: had we defined $\ell_{d+1} = \ell_{d+2}  = \ldots = \ell_{q} = 0$, then $b_0$ would be counting the number of zeros in this extended sequence. We opted to write our partitions as $d$-dimensional vectors in order to avoid an additional  block of $(q-d)$ zeros at the end.}, given by $b_0(\bell) := | \{ 1 \leq i \leq d \ : \ \ell_i = 0\}| + (q-d)$.  Note that $\bb(\bell)$ is a $(d+1)$-dimensional vector. We adopt the multinomial  notation
$$
\binom{d}{\bell}: = \frac{d!}{\ell_1! \, \ell_2!\, \ldots\, \ell_d!} \ \ \ {\rm and} \ \ \ \binom{q}{\bb(\bell)}: = \frac{q!}{b_0(\bell)!\,b_1(\bell)! \, b_2(\bell)!\, \ldots\, b_d(\bell)!}.
$$
We start by formulating the following conjecture.
 
\begin{conjecture}\label{Conj1}
Let $p >d$. With notations as above, the inequality 
\begin{equation}\label{20250524_04:58}
\|(f \sigma)^{\vee}\|_{L^{2d}(\F_q^d, \d \x)} \leq \left( \frac{1}{q^d}\sum_{\bell \in \P_d} \binom{d}{\bell}^2 \binom{q}{\bb(\bell)} \right)^{\!\! \frac{1}{2d}} \|f\|_{L^{2}(\Gamma, \d\sigma)}
\end{equation}
holds and is sharp. Moreover, $f$ is a maximizer of \eqref{20250524_04:58} if and only if $|f|$ is constant. 
\end{conjecture}
 
We make progress towards this conjecture in two complementary directions, namely, in low dimensions $d$ and in the regime of large field cardinality $q$. Our two main results are as follows. 

\begin{theorem}\label{Thm1}
Conjecture \ref{Conj1} holds for $2 \leq d \leq 20$.
\end{theorem}

\begin{theorem}\label{Thm2}
Conjecture \ref{Conj1} holds whenever 
$$
q \geq \frac{d(d-1)}{2 \log 6} + \frac{(2d-1)}{3}.
$$
\end{theorem}

\noindent {\sc Remark}: The case $d=2$ of Theorem \ref{Thm1} has already been settled in the recent M.Sc.\@ thesis of  P.\@ Ronda \cite{Ro25} under the supervision of the fourth author.

\smallskip

The strategy to prove   these results relies on an interesting interplay between analysis, algebra and combinatorics, and  can be divided into the following five main steps. 

\subsubsection*{Step 1. Counting solutions} The convolution structure of the problem allows for a reformulation in which one needs to understand the size of the preimage for each point in the support of the $d$-fold convolution of the normalized counting measure supported on the moment curve.

\subsubsection*{Step 2. Symmetric sums and Muirhead's inequality} The proposed inequality can then be rephrased as an inequality in terms of symmetric sums, with explicit but somewhat complicated weights. At this point, Muirhead's inequality \cite{Muirhead Thesis} emerges as the natural tool to be used.

 \subsubsection*{Step 3. Fractional matching} In order to appropriately use Muirhead's inequality, one needs to verify a finite number of conditions in the spirit of a fractional matching version of Hall's marriage theorem for bipartite graphs. The suitable result for our purposes is known as Strassen’s theorem \cite{Str65}. 

 \subsubsection*{Step 4. Computational part} At this point we have reduced the problem to a finite, but potentially large, number of verifications for each given $d$ and $q$. The case of large $q$ is taken care of by Step 5. We run our algorithm to verify the remaining conditions in low dimensions and conclude the proof of Theorem \ref{Thm1}. A subtle issue here is that, when $d$ grows and $q$ is close to $d$, the number of conditions to be verified grows rapidly; see \S \ref{Sec3.1_comp}.

 \subsubsection*{Step 5. Asymptotic methods} In arbitrary dimensions, an adequate lower bound on $q$ allows for the use of asymptotic methods to show that we are placed in a favourable situation to use Muirhead's inequality directly. This is the idea for the proof of Theorem \ref{Thm2}.
 \smallskip
 
We regard the conceptual part of this general strategy, and in particular the reduction of the original problem to a finite number of checks,  as the main contribution of the paper.
The implementation of the computational part which, in particular, yields Conjecture \ref{Conj1}
 in low dimensions ({$d\leq 20$}),  is primarily intended as a proof of concept.
Additional computational power will likely lead to the resolution of Conjecture \ref{Conj1} beyond $d=20$.
 \section{Proofs}\label{Proofs}
 
In order to simplify the notation, we will usually drop the star from the frequency space $\F_q^{d\ast}$, and let ${\bf C} :=  {\bf R}^\ast_\Gamma(2\to 2d)$.

\subsection{Counting solutions} We first observe that the extension inequality \eqref{20250526_09:29} is equivalent to the combinatorial inequality
\begin{align}\label{20250527_18:14}
\sum_{\bzeta \in \F_q^{d}} \left| \sum_{\substack{\bxi_1  + \ldots + \bxi_d = \bzeta \\ \bxi_i \in \Gamma}}\   \prod_{i=1}^d f(\bxi_i)\right|^2 \leq {\bf C}^{2d} \left( \sum_{\bxi \in \Gamma} |f(\bxi)|^2 \right)^d.
\end{align}
In fact, this is a specialization of \cite[Proposition 2.1]{CGRDOS} to our situation and, for the convenience of the reader, we briefly reproduce the proof here. Using the definition \eqref{20250801_17:50}, expanding the product, and reversing the order of summation, we get 
\begin{align*}
\sum_{\x \in \F_q^d} \big|(f \sigma)^{\vee}(\x) \big|^{2d} =  \frac{1}{q^{2d}}  \sum_{\x \in \F_q^d} \left|\sum_{\bxi \in \Gamma} f(\bxi) \, e(\x \cdot \bxi)\right|^{2d} = \frac{1}{q^{d}} \sum^{*}  \prod_{i=1}^d f(\bxi_i) \overline{f({\bm \eta}_i)}\,,
\end{align*}
where the sum $\displaystyle\sum^{*}$ runs over $d$-tuples $(\bxi_i)_{i=1}^{d}, ({\bm \eta}_i)_{i=1}^{d} \in \Gamma^d$ such that $\sum_{i=1}^d \bxi_i = \sum_{i=1}^d {\bm \eta}_i$. Note the use of the orthogonality relation
\begin{align}\label{20250801_18:08}
\sum_{\x \in \F_q^d} e \left( \x \cdot \left(\sum_{i=1}^d \bxi_i - \sum_{i=1}^d {\bm \eta}_i\right) \right) = 0
\end{align}
unless $\sum_{i=1}^d \bxi_i - \sum_{i=1}^d {\bm \eta}_i = 0$, in which case the left-hand side of \eqref{20250801_18:08} equals $q^d$. Finally, in order to arrive at \eqref{20250527_18:14}, we note that 
\begin{equation*}
\sum^{*}  \prod_{i=1}^d f(\bxi_i) \overline{f({\bm \eta}_i)} = \sum_{\bzeta \in \F_q^{d}} \left| \sum_{\substack{\bxi_1  + \ldots + \bxi_d = \bzeta \\ \bxi_i \in \Gamma}}\   \prod_{i=1}^d f(\bxi_i)\right|^2 .
\end{equation*}

\smallskip

Let $\bzeta = (\zeta_1, \ldots, \zeta_d) \in \F_q^{d}$. Our next order of business is to count the number of solutions of 
$$
\bxi_1  + \ldots + \bxi_d = \bzeta
$$
with each $\bxi_i \in \Gamma$. Writing $\bxi_i = (t_i, t_i^2, \ldots, t_i^d)$, and taking into account permutations, this amounts to counting the number of solutions  to the system
\begin{align}\label{E : Phi = zeta}
\begin{split}
t_1 + t_2 + \ldots + t_d &= \zeta_1\\
t_1^2 + t_2^2 + \ldots + t_d^2 &= \zeta_2\\
\vdots  \ \ \ \ \ \  \ & \\
t_1^d + t_2^d + \ldots + t_d^d &= \zeta_d.
\end{split}
\end{align}
Then each of the sums $s_k: = \sum_{i_1 < i_2 < \ldots < i_k} t_{i_1}t_{i_2}\ldots t_{i_k}$ is completely determined as a function of $\bzeta = (\zeta_1, \ldots, \zeta_d)$, and conversely, by Newton's power sum formulas \cite[p.~81]{Waerden53}. One sees that the set $\{t_1, t_2, \ldots, t_d\}$ must coincide with the set of roots of the polynomial
$$
P_{\bzeta}(X) := X^d - s_1X^{d-1} + s_2X^{d-2} - s_3X^{d-3}+\ldots + (-1)^ds_d.
$$
Therefore, the system \eqref{E : Phi = zeta} admits solutions if and only if the polynomial $P_{\bzeta}(X)$ factors completely over $\F_q$. Moreover, if $P_{\bzeta}(X)$ factors as
$$
P_{\bzeta}(X) = (X - r_1)^{\ell_1} (X - r_2)^{\ell_2} \ldots (X - r_k)^{\ell_k},
$$
with $r_1, r_2, \ldots, r_k \in \F_q$ being distinct roots, respectively of multiplicity $\ell_1 \geq \ell_2\geq  \ldots \geq \ell_k$, we must count the number of ways to associate the $t_i$'s to the $r_j$'s. From the total of $d$ available $t_i$'s we can choose $\ell_1$ to be $r_1$ in $\binom{d}{\ell_1}$ ways. After this choice, from the remaining $(d- \ell_1)$ available $t_i$'s we can choose $\ell_2$ to be $r_2$ in $\binom{d - \ell_1}{\ell_2}$ ways, and so on. The total number of possibilities is then
$$
\binom{d}{\ell_1}\binom{d- \ell_1}{\ell_2}\binom{d- \ell_1 - \ell_2}{\ell_3} \ldots \binom{d- \ell_1 - \ell_2 - \ldots - \ell_{k-1}}{\ell_k} = \binom{d}{\bell}.
$$
We can summarize this fact as follows. Let $\bell = (\ell_1, \ell_2, \ldots, \ell_d) \in \P_d$, and let $W_{\bell} \subset \F_q^d$ be the set of $\bzeta \in \F_q^d$ such that the polynomial $P_{\bzeta}(X)$ factors completely over $\F_q$, having roots of multiplicity $\ell_1 \geq \ell_2\geq  \ldots \geq \ell_d$ (notice that a potential string of zeros at the end of the partition is harmless). Then, if $\bzeta \in W_{\bell}$, the system \eqref{E : Phi = zeta} admits $\binom{d}{\bell}$ solutions.\footnote{This can be equivalently rewritten  as $\sigma^{\ast d}(\bzeta) =\binom{d}{\bell}$, where the $d$-fold convolution $\sigma^{\ast d}$ is defined as 
$$\sigma^{\ast d}(\bzeta) := \int_{\Gamma^d} {\bm \delta}(\bzeta - \bxi_1 - \bxi_2 - \ldots - \bxi_d) \,\d\sigma(\bxi_1)\,\d\sigma(\bxi_2)\ldots \d\sigma(\bxi_d) = \sum_{\substack{\bxi_1  + \ldots + \bxi_d = \bzeta \\ \bxi_i \in \Gamma}}\ 1\,,$$
and  ${\bm \delta}$ denotes the normalized Dirac delta in $\F_q^{d\ast}$  (i.e., ${\bm \delta}(0) = q^d$ and ${\bm \delta}(\bxi) = 0$ if $\bxi \neq 0$). Note that $\sigma^{\ast d}$ may take as many values as the distinct values of $\bell!$, and this large number of possibilities is generally regarded as  an obstacle when investigating the sharp inequality. To put this in perspective, note that the corresponding $k$-fold convolutions of the corresponding measures for the paraboloids and cones studied in \cite{CGRDOS}  take only two non-zero values in general (or three in one very particular instance).}

\smallskip

Our next order of business is to understand the cardinality of each set $W_{\bell}$. For this, recall that we defined $b_j(\bell) := | \{ 1 \leq i \leq d \ : \ \ell_i = j\}|$ for $j \in\{ 1,2,\ldots,d\}$, and $b_0(\bell) := | \{ 1 \leq i \leq d \ : \ \ell_i = 0\}| + (q-d)$. Assume that $\ell_1 \geq \ell_2 \geq \ldots \geq \ell_k \geq 1$ and $\ell_{k+1} = \ldots = \ell_d = 0$. In this case,  $b_0(\bell) = q-k$. We want to count the number of polynomials $P \in \F_q[X]$ of degree $d$ that factor completely as
$$
P(X) = \prod_{i =1}^k (X - r_i)^{\ell_i}.
$$
We have $q$ possibilities for root $r_1$, then $(q-1)$ possibilities for root $r_2$, and so on, up to $(q-k+1)$ possibilities for the root $r_k$. We must, however, divide by $b_1(\bell)! \,b_2(\bell)! \ldots b_d(\bell)!$ to discount for permutations of roots that have the same multiplicity (after all, they yield the same polynomial). This leads us to conclude that  
\begin{equation}\label{20250527_20:09}
|W_{\bell}| = \frac{q(q-1)\ldots(q-k+1)}{b_1(\bell)! \,b_2(\bell)! \ldots b_d(\bell)!} =  \frac{q!}{b_0(\bell)!\,b_1(\bell)! \,b_2(\bell)! \ldots b_d(\bell)!} = \binom{q}{\bb(\bell)}.
\end{equation}

\smallskip

Note that in the sum on the left-hand side of \eqref{20250527_18:14} we only have to consider the points $\bzeta$ that belong to the support of the $d$-fold convolution $\sigma^{*d}$, i.e. $\bzeta \in \cup_{\bell \in \P_d} W_{\bell}$. If $\bzeta$ belongs to a particular $W_{\bell}$, note that the product $\prod_{i=1}^d f(\bxi_i)$ is constant over all the $\binom{d}{\bell}$ possibilities of $\bxi_i \in \Gamma$ with $\bxi_1  + \ldots + \bxi_d = \bzeta$ (since these possibilities are just certain permutations of each other); let us denote this product by  $\pi_{\bzeta}(f)$. Hence, for the term on the left-hand side of \eqref{20250527_18:14}, we have 
\begin{align*}
\sum_{\bzeta \in \F_q^{d}} \left| \sum_{\substack{\bxi_1  + \ldots + \bxi_d = \bzeta \\ \bxi_i \in \Gamma}}\   \prod_{i=1}^d f(\bxi_i)\right|^2 = \sum_{\bell \in \P_d} \, \sum_{\bzeta \in W_{\bell}} \left| \binom{d}{\bell} \pi_{\bzeta}(f)\right|^2.
\end{align*}
Similarly, the combinatorics of expanding the $d$-th power in the right-hand side of \eqref{20250527_18:14} yields
\begin{align*}
\left( \sum_{\bxi \in \Gamma} |f(\bxi)|^2 \right)^d  = \sum_{\bell \in \P_d} \, \sum_{\bzeta \in W_{\bell}}\binom{d}{\bell} \left| \pi_{\bzeta}(f)\right|^2.
\end{align*}
Hence, our inequality \eqref{20250527_18:14} is equivalent to 
\begin{align}\label{20250527_18:52}
\sum_{\bell \in \P_d} \, \sum_{\bzeta \in W_{\bell}} \left| \binom{d}{\bell} \pi_{\bzeta}(f)\right|^2 \leq  {\bf C}^{2d} \sum_{\bell \in \P_d} \, \sum_{\bzeta \in W_{\bell}}\binom{d}{\bell} \left| \pi_{\bzeta}(f)\right|^2. 
\end{align}

\smallskip

We conclude this subsection with two remarks. Firstly, note that only $|f|$ appears in \eqref{20250527_18:52}. Secondly, if one believes that $f \equiv 1$ is a maximizer, then one can guess the value of the optimal constant ${\bf C}^{2d}$ in \eqref{20250527_18:52}. In fact, when $f \equiv 1$, recalling that the right-hand sides of \eqref{20250527_18:14} and \eqref{20250527_18:52} coincide, we find
\begin{align}\label{20250527_22:01}
\sum_{\bell \in \P_d} \binom{d}{\bell} |W_{\bell}| = q^d\,,
\end{align}
and so the optimal constant ${\bf C}^{2d}$ must then be equal to 
\begin{align}\label{20250527_20:57}
{\bf A} := \frac{1}{q^d} \sum_{\bell \in \P_d} \binom{d}{\bell}^2 |W_{\bell}|.
\end{align}
Thus, by \eqref{20250527_20:09} we are led to Conjecture \ref{Conj1} (modulo the classification of maximizers).

\subsection{Symmetric sums and Muirhead's inequality} We now want to rewrite \eqref{20250527_18:52} in terms of symmetric sums. We would like to do it in such a way that all symmetric sums are normalized, which in this case means that they have the same number of terms. Let $x_1, x_2, \ldots, x_q$ be $q$ variables and let $S_q$ be the group of permutations of $\{1,2,\ldots, q\}$. Associated to a partition $\bell = (\ell_1, \ell_2, \ldots, \ell_d) \in \P_d$, we define the $\bell$-th symmetric sum $\Sigma_{\bell}$ as
\begin{align}\label{20250528_05:50}
\Sigma_{\bell} := \sum_{\tau \in S_q} x_{\tau(1)}^{\ell_1}x_{\tau(2)}^{\ell_2}\ldots x_{\tau(d)}^{\ell_d}.
\end{align}
Note that, by definition, the symmetric sum $\Sigma_{\bell}$ has $q!$ terms. As an example, if $q=5$, $d=3$ and $\bell = (2,1,0)$, we would have
$$
\Sigma_{(2,1,0)} = \sum_{\substack{1 \leq i, j \leq 5 \\ i\neq j}}   6x_i^2 x_j.
$$
Returning to our setup, the $q$ variables will be $\{|f(\bxi)|^2\, : \bxi \in \Gamma\}$. A counting argument (consider $f \equiv 1$ in order to find the multiplicative constant) yields
\begin{align}\label{20250527_19:50}
 \sum_{\bzeta \in W_{\bell}} \left| \pi_{\bzeta}(f)\right|^2 = \frac{|W_{\bell}|}{q!} \Sigma_{\bell}.
\end{align}

\smallskip

In light of \eqref{20250527_18:52}, \eqref{20250527_20:57} and \eqref{20250527_19:50}, the desired inequality  becomes 
\begin{align*}
\sum_{\bell \in \P_d} \binom{d}{\bell}^2 |W_{\bell}| \, \Sigma_{\bell} \leq {\bf A} \sum_{\bell \in \P_d} \binom{d}{\bell} |W_{\bell}| \, \Sigma_{\bell}.
\end{align*}
This can be rewritten as
\begin{align}\label{20250527_22:29}
\sum_{\bell \in \P_d} \omega_{\bell} \, \Sigma_{\bell}\geq 0\,,
\end{align}
where the weight $\omega_{\bell}$ is defined by 
$$
\omega_{\bell} := \left( {\bf A} - \binom{d}{\bell}\right) \binom{d}{\bell}  \,|W_{\bell}|.
$$

\smallskip

We make two remarks. Firstly, from \eqref{20250527_22:01} and \eqref{20250527_20:57}, we have
\begin{equation}\label{20250527_22:33}
 \sum_{\bell \in \P_d} \omega_{\bell} = 0.
\end{equation}
Secondly, $\omega_{\bell} <0$ if and only if
 \begin{equation}\label{20250527_22:36}
 \bell ! := \ell_1!\, \,\ell_2!\,\ldots\,\ell_d! < \frac{d!}{{\bf A}}.
 \end{equation}
We proceed to decompose $\P_d$ into two disjoint subsets, 
\begin{equation}\label{eq_defNP}
\P_d = N \cup P\,,
\end{equation}
where a partition $\n = (n_1, \ldots, n_d)$ belongs to $N$ if $\omega_{\n} <0$, and a partition $\p = (p_1, \ldots, p_d)$ belongs to $P$ if $\omega_{\p} \geq 0$. Inequality \eqref{20250527_22:29} can be rewritten as
\begin{equation}\label{20250528_15:50}
\sum_{\n \in N} (-\omega_{\n}) \, \Sigma_{\n} \leq \sum_{\p \in P} \omega_{\p} \, \Sigma_{\p}. 
\end{equation}
In this way, \eqref{20250528_15:50} amounts to an inequality between sums of weighted symmetric sums, both of which have non-negative weights. In light of \eqref{20250527_22:33}, the total weight on the left-hand side is the same as on the right-hand side. According to \eqref{20250527_22:36}, the partitions in $N$ have the lowest values of $\bell!$. 

\smallskip

In the set of integer partitions, there is a notion of partial order that is suitable for our purposes. Given two partitions $\bell  = (\ell_1, \ell_2, \ldots, \ell_d)$ and $\bell'  = (\ell_1', \ell_2', \ldots, \ell_d')$ we say that $\bell \leq \bell'$ if the sequence of partial sums of $\bell$ is always majorized by the sequence of partial sums of $\bell'$, that is, 
\begin{align*}
\ell_1 + \ldots + \ell_k \leq \ell_1' + \ldots + \ell_k'
\end{align*}
for each $1 \leq k \leq d$. This is known as the {\it dominance order}. One can show that, under our standing hypothesis that the coordinates of each $\bell$ are decreasing,
\begin{equation*}
\bell \leq \bell' \Rightarrow  \bell ! \leq (\bell') !. 
\end{equation*}
We remark that  the converse does not necessarily hold, since two partitions might not be comparable (e.g. take $\bell = (4,1,1)$ and $\bell'= (3,3,0)$ in $\mathcal P_6$).

\smallskip

The following general inequality between symmetric sums arises naturally as a tool to be used in the present situation. In what follows, if $\bell = (\ell_1, \ldots, \ell_d) \in \P_d$, we set $\rho(\bell) := | \{ 1 \leq i \leq d \ : \ \ell_i >0\}|$. Note that if $\bell \leq \bell'$, then $\rho(\bell) \geq \rho(\bell')$. 

\begin{lemma}[Muirhead's inequality \cite{Muirhead Thesis}]\label{Muirhead}
Let $q \in \N$ and $x_1, x_2, \ldots, x_q$ be non-negative real numbers. For each $\bell \in \P_d$, let the symmetric sum $\Sigma_{\bell}$ be defined as in \eqref{20250528_05:50}. If $\bell \leq \bell'$ then 
\begin{equation}\label{20250528_06:14}
\Sigma_{\bell} \leq \Sigma_{\bell'}\,,
\end{equation}
with equality if and only if at least one the following conditions holds: 
\begin{itemize}
\item[(i)] $\bell = \bell'$\,$;$
\item[(ii)] $x_1 = x_2 = \ldots =x_q$\,$;$ 
\item[(iii)] At least $(q-\rho(\bell')+1)$ of the numbers $\{x_i\}_{i=1}^q$ are zero, in which case both sides of \eqref{20250528_06:14} are zero\,$;$
\item[(iv)]  $\rho(\bell) = \rho(\bell')$, with $k$ of the numbers $\{x_i\}_{i=1}^q$ being zero for a certain $k \leq q - \rho(\bell')$, and the remaining $(q-k)$ positive numbers being all equal.
\end{itemize}
\end{lemma}
\noindent {\sc Remarks:} In Lemma~\ref{Muirhead}, $q$ can be any natural number (not necessarily a power of a prime). Muirhead's inequality (also found in \cite[pp.~44--45]{HLP88}) is usually stated for positive numbers $x_1,\ldots,x_q$, with the equality cases being (i) and (ii). We arrive at the inequality in the case of non-negative $x_j$'s by continuity, but the equality cases (iii) and (iv) are a bit more subtle and, as we have not been able to find a reference in the literature, we provide a brief proof.  

\begin{proof}[Proof of Lemma~\ref{Muirhead} \textup{(}cases of equality \textup{(iii)} and \textup{(iv))}]
Consider a non-negative sequence $x_1,\ldots,x_q$. As remarked above, we may assume that some, but not all, of the $x_j$'s are equal to zero.  By symmetry, we may further assume that $x_1,\ldots,x_{q-k}$ are positive, and that $x_{q-k+1} = \cdots = x_q=0$.  

\smallskip

Observe that if $\Sigma_{\bell'} = 0$, then $\Sigma_\bell=0$ as well, and the former identity  holds if and only if each summand in $\Sigma_{\bell'}$ equals zero.  Moreover, the summands of $\Sigma_{\bell'}$ are all zero if and only if $q-k<\rho(\bell')$, i.e., (iii) holds. 

\smallskip

Finally, consider the case $q-k \geq \rho(\bell')$, so that $\Sigma_{\bell'} \neq 0$.  Of course, $\Sigma_\bell=
\Sigma_{\bell'}$ can only hold if $\Sigma_\bell \neq 0$, so we may assume that $q-k \geq \rho(\bell)$ as well. Define
$$
\widetilde\Sigma_\bell:=\sum_{\tau \in S_{q-k}} x^{\ell_1}_{\tau(1)} x_{\tau(2)}^{\ell_2} \ldots x_{\tau(d)}^{\ell_d}
$$
(with $x_{\tau(j)}^{\ell_j}$ interpreted as 1 if $\rho(\bell)<j\leq d$),
which involves only positive $x_j$'s.  
By standard combinatorial arguments,  
$$
\Sigma_\bell = \frac{(q-\rho(\bell))!}{(q-k-\rho(\bell))!} \widetilde \Sigma_\bell = (q-\rho(\bell))  \ldots  (q-k-\rho(\bell)+1) \widetilde\Sigma_\bell,
$$
and, since $\rho(\bell) \geq \rho(\bell')$, 
$$
\frac{(q-\rho(\bell))!}{(q-k-\rho(\bell))!} \leq \frac{(q-\rho(\bell'))!}{(q-k-\rho(\bell'))!} ,
$$
with equality if and only if $\rho(\bell) = \rho(\bell')$.  Applying Muirhead's inequality in the case of positive $x_j$, we arrive at the final case of equality, condition (iv).  
\end{proof}

\subsection{Fractional matching} 

\subsubsection{The strategy} \label{strategy}
With \eqref{20250528_15:50} in mind, we seek to find, for each pair $(\n, \p) \in N \times P$, a non-negative weight $\tau_{\n \p}$ such that:
\begin{itemize}
\item[(i)] If $\n$ and $\p$ are non-comparable (with respect to the partial dominance order), then $\tau_{\n \p}= 0$. 
\item[(ii)] For each $\n \in N$, setting $P_{\n} := \{ \p \in P\, :\, \n \leq \p\}$, we have
\begin{align}\label{20250529_21:03}
(-\omega_{\n}) = \sum_{\p \in P_{\n}} \tau_{\n \p}.
\end{align}
\item[(iii)] For each $\p \in P$, setting $N_{\p} := \{ \n \in N\, :\, \n \leq \p\}$, we have
\begin{align}\label{20250529_14:28}
\omega_{\p} = \sum_{\n \in N_{\p}} \tau_{\n \p}.
\end{align}
\end{itemize}
Once we succeed in this task, the sharp inequality \eqref{20250528_15:50} can then be established using \eqref{20250528_06:14}--\eqref{20250529_14:28} as follows:
\begin{align}\label{20250529_14:32}
\sum_{\n \in N} (-\omega_{\n}) \, \Sigma_{\n} = \sum_{\n \in N}  \sum_{\p \in P_{\n}} \tau_{\n \p} \, \Sigma_{\n} \leq \sum_{\n \in N}  \sum_{\p \in P_{\n}} \tau_{\n \p} \, \Sigma_{\p} = \sum_{\p \in P}  \sum_{\n \in N_{\p}} \tau_{\n \p} \, \Sigma_{\p} = \sum_{\p \in P}  \omega_{\p}\Sigma_{\p}.
\end{align}

\smallskip

For the case of equality, note that the partition $\p = (d,0,\ldots, 0) \in P$ has $\omega_{\p} >0$ and majorizes all the other partitions in $\P_d$; thus, in particular, it majorizes all the partitions in $N$. Hence, from \eqref{20250529_14:28}, there exists at least one partition $\n \in N$ such that $\tau_{\n\p} >0$. For this particular pair, inequality $\tau_{\n \p} \, \Sigma_{\n} \leq \tau_{\n \p} \, \Sigma_{\p}$ was invoked in \eqref{20250529_14:32} and we see from Lemma \ref{Muirhead} that equality occurs in this case if and only if all the variables (which are $\{|f(\bxi)|^2\, : \bxi \in \Gamma\}$ in our setup) are equal.

\subsubsection{An example} Let us show how to directly settle the case of the moment curve in $d=3$ for general $q=p^n$, $p>3$ prime and $n\geq 1$. In this case, one has $\P_3 = \{(1,1,1), (2,1,0), (3,0,0)\}$. We then compute
\begin{align*}
|W_{(1,1,1)}|= \frac{q(q-1)(q-2)}{6} \ \ ; \ \ |W_{(2,1,0)}|= q(q-1) \ \ ; \ \  |W_{(3,0,0)}|= q.
\end{align*}
Then 
$$
{\bf A} = \frac{6q^3 -9q^2 + 4q}{q^3} = 6 -\frac{9}{q} + \frac{4}{q^2},
$$
and
\begin{align*}
\omega_{(1,1,1)} = \frac{(-9q+4)(q-1)(q-2)}{q} \  ;  \ \omega_{(2,1,0)} = \frac{3(3q^2 - 9q + 4)(q-1)}{q} \  ;  \ \omega_{(3,0,0)} = \frac{5q^2 - 9q + 4}{q}. 
\end{align*}
Note that $\omega_{(1,1,1)} + \omega_{(2,1,0)}  + \omega_{(3,0,0)} = 0$ and, since $q>3$, we have 
$$
\omega_{(1,1,1)} <0 \ \ ; \ \ \omega_{(2,1,0)} >0 \ \ ; \ \ \omega_{(3,0,0)}>0.
$$
Since $(1,1,1) \leq (2,1,0) \leq (3,0,0)$, our fractional matching is accomplished by setting 
\begin{align*}
\tau_{(1,1,1)(2,1,0)} := \omega_{(2,1,0)} \ \ {\rm and} \ \ \tau_{(1,1,1)(3,0,0)} := \omega_{(3,0,0)}. 
\end{align*}
This concludes the proof of Theorem \ref{Thm1} in the case $d=3$. 

\subsubsection{Strassen's theorem} It turns out that the possibility  of accomplishing the fractional matching construction proposed in \S \ref{strategy} is a well-understood situation in combinatorics, in the context of bipartite graphs. This is characterized by Strassen's theorem \cite{Str65}, which we now describe. In what follows let $G = (V,E,w)$ be a weighted graph, where $V$ is the set of vertices, $E$ is the set of (undirected) edges, and each $v \in V$ is assigned a weight $w(v)$. For any subset $U \subset V$ we set $w(U) := \sum_{v \in U} w(v)$. Recall that a bipartite graph is a graph in which the vertices can be partitioned into two sets $\{A, B\}$ such that all edges have one endpoint in $A$ and the other endpoint in $B$.

\begin{lemma}[Strassen’s theorem] \label{Strassen}
Let $G = (V, E, w)$ be a weighted bipartite graph with bipartition $\{A, B\}$ such that $w(A) = w(B)$. Then the following conditions are equivalent:
\begin{itemize}
\item[(i)] For all $U \subset A$ we have $w(U) \leq w({\mc N}(U))$, where ${\mc N}(U)$ denotes the set of neighbors of $U$, i.e. the set of all vertices in $B$ that are connected by an edge to some vertex in $U$.

\smallskip

\item[(ii)] There exists an edge weight function $\widetilde{w}:E \to [0,\infty)$ such that, for all $x \in V$, we have $w(x) = \sum_{e \sim x} \widetilde{w}(e)$, where the sum is taken over all edges $e$ incident to $x$.
\end{itemize}
\end{lemma}

\subsubsection{Reduction to finitely many checks} Our situation plainly falls under the scope of Lemma \ref{Strassen}, with the vertices being the elements of $\P_d$, and the bipartition of the vertices given by $\P_d = N \cup P$. The weight of a given $\n \in N$ is $-\omega_{\n}$, and the weight of a given $\p \in P$ is $\omega_{\p}$ (hence both sets $N$ and $P$ have the same total weight). We declare that there is an edge connecting $\n \in N$ and $\p \in P$ if and only if $\n \leq \p$. Therefore, by Lemma \ref{Strassen}, our strategy boils down to checking a finite number of inequalities, namely: for any $U \subset N$, we must have
\begin{align}\label{20250529_21:12}
\sum_{\n \in U} (-\omega_{\n}) \leq \sum_{\p \, \in \underset{\n \in U}{\bigcup}  P_{\n}} \omega_{\p}.
\end{align}

\subsection{Computational part} For each given $d$ and $q$, the success of our strategy depends on \eqref{20250529_21:12}, and can be verified in finite time. For a fixed dimension $d$, if one wants to establish the result for all $q$, we need a tool to prove it for large $q$, i.e. for $q \geq c(d)$, and hence the success of our strategy is reduced again to the verification of \eqref{20250529_21:12} for $d < q < c(d)$. Such a tool is established in Theorem \ref{Thm2}, which we prove in the following subsection. The computational verification of Theorem \ref{Thm1} is then described in Appendix \ref{appendix}.

\subsection{Asymptotic methods} In this subsection, we prove Theorem \ref{Thm2}. The main idea is the following: looking at \eqref{20250527_22:36}, if $d!/{\bf A} \leq 6$, then the set $N$ will contain at most three partitions, namely $(1,1,\ldots, 1), (2,1,\ldots, 1,0)$ and $(2,2,1,\ldots, 1,0,0)$. Observe that 
$$
(1,1,\ldots, 1) \leq (2,1,\ldots, 1,0) \leq (2,2,1,\ldots, 1,0,0)\,,
$$
and that these three partitions are less than or equal to any other of the remaining partitions. Hence, in this case, each partition in $N$ will be less than or equal to each partition in $P$, and \eqref{20250529_21:12} clearly holds.

\smallskip

In light of \eqref{20250527_20:57}, the condition $d!/{\bf A} \leq 6$ is equivalent to 
\begin{align}\label{20250529_21:41}
\frac{d!}{6}\, q^d \leq \sum_{\bell \in \P_d} \binom{d}{\bell}^2 |W_{\bell}|.
\end{align}
We shall consider only one summand on the right-hand side of \eqref{20250529_21:41}, namely, the one associated to the partition $\bell = (1,1,\ldots, 1)$ (note that this is the only one of order $q^d$). Hence \eqref{20250529_21:41} follows from
\begin{align*}
\frac{d!}{6}\, q^d \leq (d!)^2 \big|W_{(1,1,\ldots,1)}\big| = d!\, q(q-1)\ldots(q-d+1)\,,
\end{align*}
where we have used \eqref{20250527_20:09}. In turn, this is equivalent to 
$$
\frac{1}{6} \leq 1 \left(1 - \frac{1}{q}\right) \left(1 - \frac{2}{q}\right) \ldots \left(1 - \frac{d-1}{q}\right)\,,
$$
and taking logarithms, we get 
\begin{align}\label{20250529_22:14}
\log 6 \geq - \sum_{j=1}^{d-1}\log \left(1 - \frac{j}{q}\right). 
\end{align}

We may assume that $q \geq 2d$. Using that 
$$
x + x^2 \geq - \log (1 - x)
$$
for $0 \leq x \leq 1/2$, inequality \eqref{20250529_22:14} follows from 
\begin{align*}
\log 6 \geq \sum_{j=1}^{d-1} \left(\frac{j}{q}\right) + \sum_{j=1}^{d-1} \left(\frac{j}{q} \right)^2 = \frac{d(d-1)}{2q} +  \frac{(d-1)d(2d-1)}{6q^2}\,,
\end{align*}
which is equivalent to the quadratic equation
\begin{align}\label{20250529_22:20}
6(\log 6) q^2 - 3d(d-1)q - (d-1)d(2d-1) \geq 0.
\end{align}
Inequality \eqref{20250529_22:20} is verified if
\begin{equation}\label{20250529_22:26}
q \geq \frac{3d(d-1) \left(1 + \sqrt{1 + \frac{8\log6}{3} \frac{(2d-1)}{d(d-1)}}\right)}{12 \log 6}.
\end{equation}
Using the fact that $\sqrt{1 +x} \leq 1 + \frac{x}{2}$ for any $x\geq 0$, inequality \eqref{20250529_22:26} then follows from 
\begin{align*}
q \geq \frac{3d(d-1) \left(2 + \frac{4\log6}{3} \frac{(2d-1)}{d(d-1)}\right) }{12 \log 6} = \frac{d(d-1)}{2 \log 6} + \frac{(2d-1)}{3}.
\end{align*}
This completes the proof of Theorem \ref{Thm2}.

\section{Concluding remarks}

\subsection{On  computational complexity} \label{Sec3.1_comp} The Hardy--Ramanujan formula \cite{HR} provides an asymptotic formula for the number of partitions in $\P_{d}$:
\begin{equation}\label{20250805_11:29}|\P_{d}| \sim \frac{1}{4d\sqrt{3}}e^{\pi \sqrt{2d/3}},
\text{ as } d \to \infty.
\end{equation}
 The computational work needed to determine the weights $\omega_{\bell}$ thus grows as in \eqref{20250805_11:29}. However,   a complete computation is not always necessary, as shown in the proof of Theorem \ref{Thm2}. In that regime, the size of the set $N = \{\bell \in \P_d\, : \, \omega_{\bell} <0\}$ is small, which allowed for a direct verification  of \eqref{20250529_21:12}. 

\smallskip

In this subsection, we briefly explain why the size of the power set of $N$ (i.e., the number of conditions to be checked in \eqref{20250529_21:12}) increases rapidly  as $d$ increases and $q$ remains close to $d$. Even though in our setup the smallest possible value of $q$ equals $d+1$, in the asymptotic argument below it is harmless to assume, as we shall, that  $q=d$. From \eqref{20250527_20:57} we then have
\begin{align}\label{20250802_10:21}
{\bf A} = \frac{1}{d^d} \sum_{\bell \in \P_d} \binom{d}{\bell}^2 |W_{\bell}|.
\end{align}
We seek an upper bound for ${\bf A}$. Split $\P_d$ into two disjoint subsets, $\P_d = I_{d,k} \cup I_{d,k}^\complement$, where $I_{d,k}$ denotes the set of partitions with a block of at least $k$ consecutive ones ($k$ will be chosen later). By removing the final $k$ ones in this block,  the partitions in $I_{d,k}$ are in bijective correspondence with the partitions in $\P_{d-k}$. Using  $\binom{d}{\bell} \leq d!$ and $|W_{\bell}| \leq \frac{d!}{k!} \leq d^{d-k} $ for each $\bell \in I_{d,k}$, we have
\begin{equation}\label{20250804_14:41}
\frac{1}{d^d} \sum_{\bell \in I_{d,k}} \binom{d}{\bell}^2 |W_{\bell}| \leq \frac{(d!)^2}{d^d} \cdot d^{d-k} \cdot |\P_{d-k}|.
\end{equation}
On the other hand, if $\bell \in I_{d,k}^\complement$, note that 
\begin{equation}\label{20250804_14:34}
\bell! \geq 2^{\frac{d-k}2}.
\end{equation}
In fact,  if we remove the final block of ones in $\bell$, which in this case is of length at most $k-1$, we are left with $\ell_1 + \ldots + \ell_j = a \geq d-k+1$ elements, with $\ell_1 \geq \ldots \geq \ell_j \geq 2$. Using that  $(m+1)!(n-1)! \geq m!n!$ if $m \geq n\geq 1$, one can check that in this case the minimal value of $\ell_1!\ldots \ell_j!$  occurs when they all equal $2$ (or one of them equals $3$, if $a$ is odd). In all such cases, \eqref{20250804_14:34} holds. From \eqref{20250804_14:34} and  \eqref{20250527_22:01},
\begin{equation}\label{20250804_14:42}
 \frac{1}{d^d} \sum_{\bell \in I_{d,k}^\complement} \binom{d}{\bell}^2 |W_{\bell}| \leq     \frac{1}{d^d} \frac{d!}{2^{(d-k)/2}}\sum_{\bell \in I_{d,k}^\complement} \binom{d}{\bell} |W_{\bell}| \leq \frac{1}{d^d} \frac{d!}{2^{(d-k)/2}}\sum_{\bell \in \P_d} \binom{d}{\bell} |W_{\bell}| = \frac{d!}{2^{(d-k)/2}}.
\end{equation}
Using \eqref{20250802_10:21}, \eqref{20250804_14:41} and \eqref{20250804_14:42} we arrive at 
\begin{equation}\label{20250804_15:06}
    \frac{{\bf A}}{d!} \leq \frac{d!}{d^d}\cdot d^{d-k}\cdot |\P_{d-k}| + \frac{1}{2^{(d-k)/2}}.
\end{equation}

\smallskip

We now choose \footnote{Strictly speaking, we should consider  the integer part of this number, but this is again harmless for the asymptotic argument.} $k  = d - \frac{d}{2\log d}$.  Stirling's formula 
together with \eqref{20250805_11:29} in \eqref{20250804_15:06} yields 
\begin{align*}
  \frac{{\bf A}}{d!}& \lesssim \frac{\sqrt{2\pi d} }{e^{d}} \cdot d^{\frac{d}{2\log d}} \cdot  \frac{e^{\pi \sqrt{d/(3(\log d))}}}{ \frac{2d}{(\log d)}\sqrt{3}} + \frac{1}{2^{\frac{d}{4\log d}}}\lesssim  \frac{1}{2^{\frac{d}{4\log d}}}.
\end{align*}

\smallskip

Recall condition \eqref{20250527_22:36}. Given $a \geq 1$, the partition $\bell = (a,1,1,\ldots,1,0,0,\ldots,0)$ belongs to $N$ whenever 
$$a! \lesssim 2^{\frac{d}{4\log d}}.
$$
By Stirling's formula, this is guaranteed for $a \sim \frac{d}{8(\log d)^2}$ and $d$ large. Then, any partition of the form $(\ell_1, \ell_2, \ldots, \ell_j, 1,1,\ldots 1, 0, \ldots, 0)$, where $\ell_i \geq 1$ and $\sum_{i=1}^j \ell_i = a$, belongs to $N$ as well. There are as many such partitions as \[|\P_a| \sim \frac{2(\log d)^2}{d\sqrt{3}}e^{\pi \sqrt{d/(12(\log d)^2)}}.\] It follows that  the size of the power set of $N$, which corresponds to the number of conditions to be verified in \eqref{20250529_21:12}, grows at least as fast as 
\[2^{\frac{2(\log d)^2}{d\sqrt{3}}e^{\pi \sqrt{d/(12(\log d)^2)}}}, \text{ as } d \to \infty.\]

\subsection{Additional variables} For $k \in \N$ we might consider a curve $\Gamma \subset \F_q^{d+k}$ of the form $\Gamma = \{(\xi, \xi^2, \ldots, \xi^d, \phi(\xi)) : \xi \in \F_q\}$, where $\phi : \F_q \to \F_q^k$ is an arbitrary function, and address the sharp inequality
\begin{equation*}
\|(f \sigma)^{\vee}\|_{L^{2d}(\F_q^{d+k}, \d \x)} \leq {\bf C}_{d,q}^{(k)} \,\|f\|_{L^{2}(\Gamma, \sigma)}.
\end{equation*}
For $\bzeta = (\zeta_1, \ldots, \zeta_{d+k}) \in \F_q^{d+k}$, in analogy to \eqref{E : Phi = zeta}, one starts by understanding the structure of the solutions of the system
\begin{align*}
\begin{split}
t_1 + t_2 + \ldots + t_d &= \zeta_1\\
t_1^2 + t_2^2 + \ldots + t_d^2 &= \zeta_2\\
\vdots  \ \ \ \ \ \  \ & \\
t_1^d + t_2^d + \ldots + t_d^d &= \zeta_d\\
\sum_{i=1}^d \phi(t_i) & = (\zeta_{d+1}, \ldots, \zeta_{d+k}).
\end{split}
\end{align*}
Note that the set $\{t_1, t_2, \ldots, t_d\}$ is completely determined from the sub-vector $(\zeta_1, \zeta_2, \ldots, \zeta_d)$, and this in turn uniquely determines the values of $\zeta_{d+1}, \ldots, \zeta_{d+k}$. The analysis of the sharp inequality is then exactly the same as done in Section \ref{Proofs} and we conclude that the optimal constant ${\bf C}_{d,q}^{(k)}$ coincides with  the optimal constant of \eqref{20250526_09:29}.

\subsection{Related curves} Our strategy allows us to treat the same problem associated to other curves $\eta:\F_q \to \F_q^d$, provided that in the sum
$$
\eta(t_1)+\cdots+\eta(t_d) = \bzeta,
$$
the multiset $\{t_1,\ldots,t_d\}$ is uniquely determined by $\bzeta$.  For example, let $A$ be an invertible $d \times d$ matrix. With $\gamma (\xi) = (\xi, \xi^2, \ldots, \xi^d)$, we can consider
$$
\eta(\xi) := \gamma(\xi) \cdot A
$$
for which the optimal constant is again equal to the optimal constant of \eqref{20250526_09:29}. 

\smallskip

The analogous sharp form of~\eqref{20250526_09:29} (e.g. \cite[Theorem 1.1]{Koh2012}) for a general polynomial is an interesting question when the components of the associated polynomial map $\Phi : \F_q^d \to \F_q^d$ given by
$$
\Phi (t_1, \ldots, t_d) := \gamma (t_1) + \ldots + \gamma (t_d)
$$
do not uniquely identify $t_i$ up to permutations. For example, associated with the polynomial $\gamma (\xi) = (\xi, \xi^3, \xi^5)$, the following system of equations
\begin{align*}
\begin{split}
t_1 + t_2 + t_3 &= 1\\
t_1^3 + t_2^3 + t_3^3 &= 1\\
t_1^5 + t_2^5 + t_3^5 &= 1\\
\end{split}
\end{align*}
contains $(1, t_2, - t_2)$ within the solution set. Moreover, in the continuous setting, the dichotomy between existence of maximizers and concentration of maximizing sequences at (e.g.) $t=\pm 1$ has not been resolved \cite[Theorem 1.3]{BS25}. Thus  the methods of the present paper cannot be directly adapted.

\section*{Acknowledgements}
This project was initiated at a SQuaRE at the American Institute of Mathematics. The authors thank AIM for providing a supportive and mathematically rich environment.
Oliveira e Silva  is partially supported by FCT/Portugal through project UIDB/04459/2020 with DOI identifier 10-54499/UIDP/04459/2020.  Stovall is partially supported by NSF DMS-2246906, and Tautges is partially supported by NSF DMS-2037851.

\appendix
\section{Computational part}\label{appendix}
In this appendix, we describe the algorithm used to verify the validity of \eqref{20250529_21:12} for dimensions $2 \leq d \leq 20$ and exponents $q< \frac{d(d-1)}{2 \log 6} + \frac{(2d-1)}{3}$. 
Computations were performed in Python 3 with SageMath libraries.

\smallskip

For every admissible $d$ (line 8), we use Sage to generate the partition poset with the dominance order (line 10) and compute $\bell!$ for all $\bell \in \P_d$  (lines 11--13).
For every admissible $q$ (line 14), we compute the value of $\omega_\bell$ for all $\bell \in \P_d$ (lines 16--24).
For each fixed pair $(d, q)$, we iterate over subsets $I \subset N$ (recall that $N$ was defined in \eqref{eq_defNP}) and verify that
\[
    \sum_{\bell' \geq \bell \text{ for some } \bell \in I} \omega_{\bell'} \geq 0
\]
(lines 25--32).
This is equivalent to \eqref{20250529_21:12} and, since the check succeeds,  Conjecture~\ref{Conj1} holds for every dimension $d$ and exponent $q$ in the claimed range.
This concludes the proof of Theorem \ref{Thm1}.

\begin{lstlisting}[language=Python]
from sage.all import *
def choose(x, l):
    return factorial(x)/product([ factorial(i) for i in l ])
def get_w(l, d, q):
    c = [ l.count(i) for i in range(d+1) ]
    c[0] = q - len(l)
    return choose(q, c)
for d in range(2, 20):
    broken = False
    P = posets.IntegerPartitionsDominanceOrder(d)
    chs = {}
    for p in P:
        chs[p] = choose(d, list(p))
    for q in range(d+1, floor(d*(d-1)/(2*log(6)) + (2*d - 1)/3) + 1):
        q = Integer(q)
        ws = omegas = {}
        num = denom = 0
        for p in P:
            ws[p] = get_w(list(p), d, q)
            num += chs[p]**2 * ws[p]
            denom += chs[p] * ws[p]
        C = num / denom
        for p in P:
            omegas[p] = chs[p] * ws[p] * (C - chs[p])
        LHS = [ p for p in P if omegas[p] < 0 ]
        for I in Subsets(LHS):
            if len(I) == 0:
                continue
            if sum([ omegas[p] for p in P.order_filter(I) ]) < 0:
                print("Proposition failed at " + str(p) + ", q = " + str(q))
                broken = True
                break      
    if not broken:
        print("Proposition true for d = " + str(d))
\end{lstlisting}


\begin{thebibliography}{99}
\bibitem{BS23}
C. Biswas and B. Stovall, 
\newblock Existence of extremizers for Fourier restriction to the moment curve,
\newblock Trans. Amer. Math. Soc. {\bf 376} (2023), no. 5, 3473--3492.

\bibitem{BS25}
C. Biswas and B. Stovall, 
\newblock Sharp Fourier restriction to monomial curves,
\newblock Math. Ann. {\bf 391} (2025), no.~4, 5391--5425.

\bibitem{BDG16}
J. Bourgain, C. Demeter and L. Guth, 
\newblock Proof of the main conjecture in Vinogradov's mean value theorem for degrees higher than three,
\newblock Ann. of Math. (2) {\bf 184} (2016), no. 2, 633--682.

\bibitem{Carbery2008}
A. Carbery, B. Stones and J. Wright, 
\newblock Averages in vector spaces over finite fields, 
\newblock Math. Proc. Cambridge Philos. Soc. {\bf 144} (2008), no.~1, 13--27.

\bibitem{Dr85}
S. Drury, 
\newblock Restrictions of Fourier transforms to curves,
\newblock Ann. Inst. Fourier (Grenoble) {\bf 35} (1985), no. 1, 117--123.

\bibitem{FOS17}
D. Foschi and D. Oliveira e Silva, 
\newblock Some recent progress on sharp Fourier restriction theory,
\newblock Anal. Math. {\bf 43} (2017), no. 2, 241--265.

\bibitem{CGRDOS}
C. Gonz\'{a}lez-Riquelme and D. Oliveira e Silva,
\newblock Sharp extension inequalities on finite fields,
\newblock preprint (2024).

\bibitem{GLYZK21}
S. Guo, Z. Li, P.-L. Yung and P. Zorin-Kranich, 
\newblock A short proof of $\ell^2$ decoupling for the moment curve,
\newblock Amer. J. Math. {\bf 143} (2021), no. 6, 1983--1998.

\bibitem{HLP88}
G. H. Hardy, J. E. Littlewood and G. Pólya,
\newblock {\it Inequalities},
\newblock Reprint of the 1952 edition
Cambridge Math. Lib.
Cambridge University Press, Cambridge, 1988.

\bibitem{HR}
G. H. Hardy and S. Ramanujan,
\newblock Asymptotic Formulae in Combinatory Analysis,
\newblock Proc. London Math. Soc. (2) {\bf 17} (1918), 75--115.

\bibitem{Koh2012}
D. Koh and C. Y. Shen, 
\newblock Sharp extension theorems and Falconer distance problems for algebraic curves in two dimensional vector spaces over finite fields,
\newblock Rev. Mat. Iberoam. {\bf 28} (2012), no.~1, 157--178.




\bibitem{Mat02}
\newblock J. Matousek,
\newblock {\it Lectures on Discrete Geometry},
Grad. Texts in Math., 212
Springer-Verlag, New York, 2002. 

\bibitem{Mat03}
\newblock J. Matousek,
\newblock{\it Using the Borsuk--Ulam Theorem.} Lectures on topological methods in combinatorics and geometry. Written in cooperation with Anders Björner and Günter M. Ziegler.
Universitext,
Springer-Verlag, Berlin, 2003.

\bibitem{MT}
G. Mockenhaupt and T. Tao, 
\newblock Restriction and Kakeya phenomena for finite fields,
\newblock Duke Math. J. {\bf 121} (2004), no. 1, 35--74.

\bibitem{Muirhead Thesis}
R.~F.~Muirhead, 
\newblock {\it Some Methods Applicable to Identities and Inequalities of Symmetric Algebraic Functions of $n$ Letters}, 
\newblock Thesis (D.Sc.)–University of Glasgow (United Kingdom)
ProQuest LLC, Ann Arbor, MI, 1904. 34 pp.

\bibitem{NOST23}
G. Negro, D. Oliveira e Silva and C. Thiele,
\newblock When does $e^{-|\tau|}$ maximize Fourier extension for a conic section?
\newblock Harmonic analysis and convexity, 391--426.
Adv. Anal. Geom., 9
De Gruyter, Berlin, 2023.


\bibitem{Ro25}
P. Ronda,
\newblock Finite field sharp extension inequalities for curves and quadratic surfaces,
\newblock M.Sc.\@ thesis, IST, in preparation (2025).

\bibitem{Ro51}
K. F. Roth,
\newblock On a problem of Heilbronn, \newblock 
J. London Math. Soc. {\bf 26} (1951), 198--204.


\bibitem{Str65}
V. Strassen, 
\newblock The existence of probability measures with given marginals,
\newblock Ann. Math. Statist. {\bf 36} (1965), 423--439.

\bibitem{Waerden53}
B. L. van der Waerden, 
\newblock {\it Modern Algebra}, 
\newblock Frederick Ungar, New York, 1953.




\end{thebibliography}
\end{document}